\documentclass[11pt,a4paper]{amsart}
\usepackage{geometry}     
\usepackage{amssymb}

\newtheorem{theorem}{Theorem}[section]
\newtheorem{corollary}[theorem]{Corollary} 
\newtheorem{lemma}[theorem]{Lemma}

\newtheorem{prop}[theorem]{Proposition}

\newtheorem{letterthm}{Theorem}

\theoremstyle{definition}

\newtheorem{example}[theorem]{Example}
\newtheorem{remark}[theorem]{Remark}

\def\N{\mathbb N}
\def\Z{\mathbb Z}
\def\R{\mathbb R}
\def\<{\langle}
\def\>{\rangle}

\def\mod{{\rm{Mod}}(S)} 
\def\mods{{\rm{Mod}}(S')} 
\def\e{\varepsilon}
\def\G{\Gamma}
\def\g{\gamma}

\title{On the Subgroups of Right Angled Artin Groups and Mapping Class Groups}
 
\begin{document}

\begin{abstract}  
There exist right angled Artin groups $A$ such that the isomorphism problem for finitely presented subgroups
of $A$ is unsolvable, and for certain finitely presented 
subgroups the conjugacy and membership problems are unsolvable. It follows that 
if $S$ is a surface of finite type and the genus of $S$ is sufficiently large, then
the corresponding decision problems for the mapping class group
$\mod$ are unsolvable. Every virtually special
group embeds in the mapping class group of infinitely many closed surfaces. Examples are given of
finitely presented subgroups of mapping class groups that have infinitely
many conjugacy classes of torsion elements.
\end{abstract}

\author[Bridson]{Martin R.~Bridson}
\address{Martin R.~Bridson\\
Mathematical Institute \\
24-29 St Giles' \\
Oxford OX1 3LB \\
U.K. }
\email{bridson@maths.ox.ac.uk} 

\thanks{This worked was funded by an EPSRC Senior Fellowship, a Royal Society Wolfson Merit Award, and the
Mittag Leffler Institute.}



\maketitle

\section{Introduction} 

The subgroups of mapping class group have long been the subject of intense study, and much is known about
their structure, cf. \cite{ivanov}. But until recently, there have been very few results describing 
large classes of groups that embed in mapping class groups. A symptom of this
paucity is that it has been unknown whether the finitely presented subgroups of mapping class groups 
might be constrained to such an extent that the basic decision problems of group theory are solvable in this
class. 

Recent advances in geometric group theory and low-dimensional topology allow one to construct a much
greater array of subgroups than hitherto.  The purpose of this note is to explain how one can exploit these advances to
resolve many of the problems highlighted by Farb  in
his account of the outstanding problems concerning subgroups of mapping class groups \cite{farb}.  In particular we
prove the following theorems. Here, and throughout this article, $S$ denotes a surface of finite type (which may
have non-empty boundary and punctures) and $\mod$ is its mapping class group.

\begin{letterthm}\label{t:iso} If the genus of $S$ is sufficiently
large, then the isomorphism problem for the finitely presented
subgroups of $\mod$ is unsolvable.
\end{letterthm}

\begin{letterthm}\label{t:cong} If the genus of $S$ is sufficiently
large, then there is a finitely presented subgroup of $\mod$ with unsolvable conjugacy problem. 
\end{letterthm}

\begin{letterthm}
If the genus of $S$ is sufficiently
large, then there are finitely presented subgroups of $\mod$ for which the membership problem is unsolvable.
\end{letterthm}

It is important to note that these theorems are about finitely {\em presented} subgroups: the corresponding
theorems for finitely generated groups are much easier; they
follow from the observation that the mapping class group of any 
surface of genus at least 2 contains a direct product of non-abelian free groups
(cf.~\cite{cfm-thesis}).

A {\em right angled Artin group} (RAAG) is a group given by a presentation of the form
$$
A = \< v_1,\dots,v_n \mid [v_i,v_j] = 1 \,\forall (i,j)\in E\>.
$$
A vast array of subgroups of mapping class groups arise from the observation that
an arbitrary  RAAG can be embedded in the mapping
class groups of any surface of sufficiently high genus (Section \ref{s:raags}). We
exploit this by deducing the above results from the corresponding results for RAAGs.

\begin{theorem}\label{t:a'} There exists a right angled Artin group $A$
such that the isomorphism problem for the finitely presented
subgroups of $A$ is unsolvable.
\end{theorem}

\begin{theorem}\label{t:b'}
There exists a right angled Artin group $A$ and a finitely presented subgroup $H<A$ such
that the conjugacy and membership problems are unsolvable for $H$.  
\end{theorem}
 
With more care, one can arrange for the embedding of the RAAG to lie in the Torelli subgroup 
$\mathcal T(S) < \mod$  (cf.~Koberda~\cite{koberda}), so  the undecidability phenomena
described above are already present among the finitely presented subgroups of $\mathcal T(S)$. 

Our proofs of the above results rely on the 1-2-3 Theorem of  \cite{bbms} and Bridson and Miller's work on the
isomorphism problem for finitely presented subdirect products of hyperbolic groups
\cite{BM1}. They also rely  on
Haglund and Wise's version of the Rips construction \cite{HW}.

 Theorems \ref{t:a'} and \ref{t:b'}
  stand in sharp contrast with what is known
about decision problems for finitely presented subgroups of the most classical right angled Artin
groups, i.e. direct products of free groups. In that setting, all finitely presented
subgroups have a solvable conjugacy problem and there is a uniform solution to the membership
problem \cite{BM2}. Moreover, these positive results extend to all residually free groups \cite{bhms}. The
isomorphism problem for finitely presented subgroups of direct products of free groups remains open.

Theorems A to C lend substance to the general observation that the recent developments
around the notion of virtually special groups show that the range of finitely presented subgroups of
mapping class groups is much greater than might previously have been imagined. In
Section \ref{s:raags} we shall see that
if a group virtually embeds in a RAAG, then it embeds
in the mapping class groups of infinitely many closed surfaces (but not {\em all} surfaces of large genus).
This includes all groups that are
{\em virtually special} in the sense of Haglund and Wise \cite{HW}. This
 class of groups has recently been
shown to be incredibly rich. Most spectacularly, it includes the fundamental group of every hyperbolic
3-manifold \cite{berg-wise}, \cite{agol}.

\begin{prop}\label{p:vs} If $\G$ is virtually special, then it embeds in the mapping class group of
infinitely many closed surfaces.
\end{prop}

The constructions in Section \ref{s:fin} respond to
Problem 3.10  of Farb's list \cite{farb}. We construct finitely presented subgroups
of  mapping class groups in which there are
 infinitely many conjugacy classes of elements of finite order.
(These constructions do not involve RAAGs.)  The first construction 
of subgroups of this sort was given by Brady, Clay and Dani \cite{BCD}. 
A construction that begins with ideas of Leary and Nuncinkis \cite{LN} allows one to
construct subgroups of this sort that are of type VF; see remark \ref{r:LN}.
 Note that, unlike Theorems A to C, these results do not carry over to the Torelli group, because that is torsion-free.

 In the final
 section of the paper we discuss the {\em{Dehn functions}}
 of finitely presented subgroups of RAAGs and mapping class groups.
 We give explicit examples of  right-angled Artin groups $A$ and finitely presented subgroups $P<A$
such that $P$ has an exponential Dehn function.

\section{Preliminaries}

It is assumed that the reader is familiar with basic definitions concerning
mapping class groups. We recall the basic vocabulary associated to decision problems
and recommend \cite{cfm} for a pleasant introduction to the subject.

The {\em word problem}
for a finitely generated group $\G$ asks for an algorithm that,
on input an arbitrary  word in the free group on the generators, will
correctly determine whether or not that word equals the identity in the group.
The {\em conjugacy problem} asks for an algorithm  that,
on input a pair of words, will
correctly determine whether or not they define conjugate elements of the group.
The {\em membership problem} for a subgroup $H<\G$ asks for an algorithm  that,
on input an arbitrary word in the free group on the generators of $\G$ will
correctly determine whether or not the element it defines lies in $H$.

The (in)solubility of each of these problems is independent of the chosen generating
set. The solubility
of the word problem is inherited by finitely generated subgroups and insolubility is
inherited by finite-index subgroups. 

\begin{remark}\label{r:HAG} Let $H<A<G$ be groups with $A$ and $G$ finitely generated. A solution
to the membership problem for $H$ in $G$ provides a solution to the membership problem for $H$ in $A$.
\end{remark}

The {\em isomorphism problem} for a class of finitely presentable groups asks
for a uniform partial algorithm that, 
given a pair of finite presentations defining groups in the
class, will halt and determine whether or not the groups presented are isomorphic.  

\section{The conjugacy and membership problems}

In the course of the last fifteen years, many results have been proved to the effect that 
finitely presented subgroups of 
direct products of hyperbolic (and related) groups can be remarkably wild. Most of these results rely
on the template described in \cite{mb:icm}: given a finitely presented group of interest, $Q$,
 one uses a variant of the Rips construction \cite{rips}
to construct an epimorphism $p:\G\to Q$ with finitely generated kernel, where $\G$ is hyperbolic; one then
forms the {\em{fibre product}} $P=\{(x,y)\mid p(x)=p(y)\}<\G\times\G$. In general, $P$ will be finitely generated
but not finitely presented. However, if $Q$ is of {\em type $F_3$} (i.e.~has a classifying space $K(Q,1)$ with finite 3-skeleton), then
the 1-2-3 Theorem of \cite{bbms} implies that $P$ is finitely presentable.  And if $Q$ is a complicated
group, its complications ought to transfer to $P$ in a manner one hopes to understand.

We recall the statement of the 1-2-3 Theorem.

\begin{theorem}{\cite{bbms}} Let $1\to N\to \G\overset{p}\to Q\to 1$ be a short exact sequence of groups.
If $N$ is finitely generated, $\G$ is finitely presented, and $Q$ is of type $F_3$, then
the associated fibre product is finitely presented.
\end{theorem}

The version of the Rips construction that we need here is due to Haglund and Wise. 
The only thing that the reader needs to know
about {\em virtually special} groups is that they contain subgroups of finite
index that inject into right angled Artin groups. 

\begin{theorem}{\cite{HW}}\label{t:rips}
For every finitely presented group $Q$, there exists a virtually
special hyperbolic group $H$ and an epimorphism $H\to Q$ with finitely generated kernel.
\end{theorem}

\subsection{The Membership Problem}

\begin{lemma}\label{l:memb}
Let $1\to N\to \G \overset{p}\to Q\to 1$ be a short exact sequence of groups, with $\G$
finitely generated, and let
$P<\G\times\G$ be the associated fibre product. If the word problem in $Q$ is unsolvable, then
the membership problem for $P<\G\times\G$ is unsolvable.
\end{lemma}

\begin{proof} We fix a finite generating set $X$ for $\G$ and work with the generators
$X'=\{(x,1), (1,x)\mid x\in X\}$ for $\G\times\Gamma$. Given a word $w=x_1\dots x_n$ in the free group on $X$, 
we consider the word $(x_1,1)\dots (x_n,1)$ in the free group on $X'$.
 This word defines an element of $P$ if and only if $p(w)=1$ in $Q$, and we are assuming that
 there is no algorithm that
 can determine which words in the symbols $p(x)$ equal the identity in $Q$.
 \end{proof}

\begin{theorem}\label{t:memb}
There exists a right angled Artin group $A$ and a finitely presented subgroup $P_0<A$
for which the membership problem is unsolvable.
\end{theorem}

\begin{proof}
Let $Q$ be a group of type $F_3$ that has an unsolvable word problem. (Collins and Miller \cite{CM}
gave explicit constructions of such groups.) We construct a short exact sequence
$1\to N\to \G\overset{p}\to Q\to 1$ as in Theorem \ref{t:rips}. Since $\G$ is virtually special,
it contains a subgroup of finite index $\G_0<\G$ that is a subgroup of a right angled Artin group.
Let $Q_0=p(\G_0)$ and let $N_0=\G_0\cap N$. Then we have a short exact sequence $1\to N_0
\to \G_0\to Q_0\to 1$. A subgroup of finite index in a group of type $F_n$ is again of type $F_n$,
so $N_0$ is finitely generated, $\G_0$ is finitely presented, and $Q_0$ is of type $F_3$. 
Thus the 1-2-3 Theorem implies that the fibre product $P_0<\G_0\times\G_0$ associated to this
sequence is finitely presented.

A subgroup of finite index in a group with an unsolvable word problem again has an unsolvable word
problem, so by Lemma \ref{l:memb} the membership problem for $P_0<\G_0\times\G_0$  is unsolvable.
Since $\G_0$ embeds in a right angled Artin group, $R$ say, $\G_0\times\G_0$ embeds in
the right angled Artin group $A=R\times R$.
Remark \ref{r:HAG} completes the proof.
\end{proof}

\subsection{The Conjugacy Problem}

The following observation is an easy consequence of the fact that centralisers of non-trivial
elements in hyperbolic groups are virtually cyclic.

\begin{lemma}\label{l1} Let $1\to N\to \G \to Q\to 1$ be a short exact sequence of groups with $\G$ 
torsion-free and hyperbolic. Let $P<\G\times\G$ be the associated fibre product.
Suppose that $Q$ is torsion-free and that $a\in N$ is non-trivial.
Then the centralizer of $(a,a)$ in $\G\times\G$ is contained in $P$.
\end{lemma}

Our interest in this observation is explained by the following lemma. This is Lemma 3.3 from \cite{bbms};
we reproduce the proof here for the convenience of the reader.

\begin{lemma}\label{l2} Let $H\subset P \subset G$ be finitely generated groups. Suppose that $H$ is normal
in $G$ and that there exists $a\in H$ whose centralizer $C_G(a)$ is contained in $P$. If there
is no algorithm to decide which words in the generators of $G$ belong to $P$, then $P$ has an unsolvable conjugacy problem.
\end{lemma}

\begin{proof} Fix finite generating sets $B$ for $G$ and $A$ to $H$. For each $b\in B$ and
$\e=\pm 1$, let $u_{b,\e}$ be a word in the free group on $A$ so that $b^\e ab^{-\e}=u_{b,\e}$ in $H$.

Given an arbitrary word $w$ in the free group on $B$, one can use the relations $b^\e ab^{-\e}=u_{b,\e}$ 
to convert $waw^{-1}$ into a word $w'$ in the letters $A$. (The length of $w'$ is bounded by an exponential
function of the length of $w$ and the rewriting process is entirely algorithmic.)

Now we ask if $w'$ is conjugate to $a$ in $P$. If there were to exist $h\in P$ such that $h^{-1}ah=w'$
then $hw\in C_G(a)\subset P$, whence $w\in P$. Thus $w'$ is conjugate to $a$ in $P$ if and only if $w\in P$,
and we are assuming that there is no algorithm that can determine in $w$ belongs to $P$.
\end{proof}

\begin{theorem}\label{t:conj}
The finitely presented group $P_0<A$ constructed in the proof of Theorem \ref{t:memb} has
an unsolvable conjugacy problem.
\end{theorem}

\begin{proof} In the proof of Theorem \ref{t:memb},
 the input group $Q$ had an aspherical presentation,
in other words it was the
 fundamental group of a compact 2-complex with contractible universal cover. A
cyclic group has cohomology in infinitely many dimensions and therefore cannot act freely on a finite-dimensional
contractible complex. It follows that $Q$ is torsion-free. Thus we are in the situation of Lemma \ref{l1}, and by
applying Lemma \ref{l2} with $N\times N$ in the role of $H$, we conclude from Theorem \ref{t:memb}
that $P_0$ has an unsolvable conjugacy problem.
\end{proof}
 
This concludes the proof of Theorem \ref{t:b'}.
 
\section{The Isomorphism problem}

In general, determining whether a class of groups has a solvable isomorphism problem is much
harder than considerations of the word and conjugacy problem. But with the virtually special version of
the Rips construction in hand, we can appeal directly to the following criterion of
Bridson and Miller \cite{BM1}. 

\begin{theorem}{\cite{BM1}}\label{t:bm} Let $1\to N\to \G\to L\to 1$ be an exact 
sequence of groups. Suppose that 
\begin{enumerate}
\item $\G$ is torsion-free and hyperbolic,
\item $N$ is infinite and finitely generated, and 
\item $L$ is a non-abelian free group.
\end{enumerate}
If $F$ is a non-abelian free group, then the isomorphism problem for
finitely presented subgroups of $\G\times\G\times F$ is unsolvable. 

In more detail, there is a recursive sequence $\Delta_i\ (i\in\N)$ of
finite subsets of $\G\times\G\times F$, together with finite presentations 
$\< \Delta_i \mid \Theta_i\>$ of the subgroups they generate, such that 
there is no algorithm that can determine 
whether or not $\<\Delta_i\mid\Theta_i\>\cong \<\Delta_0 \mid \Theta_0\>$.
\end{theorem}

The proof of this theorem does not lend itself to an easy summary,
so we refer the reader to \cite{BM1} for details. 

\subsection*{Proof of Theorem \ref{t:a'}} 
By applying the Haglund-Wise version of the Rips construction with
$Q$ a finitely generated 
non-abelian free group, we obtain a short exact sequence
$1\to N\to \G\to Q\to 1$ with $N$ finitely generated
and $\G$ hyperbolic and virtually special. 
We pass to a subgroup of finite index $\G_0<
G$ that embeds in a right angled Artin group $A$ and consider
the sequence $1\to N_0\to \G_0\to Q_0\to 1$ as in the previous
section. $N$ is finitely
generated and $Q_0$ is a finitely generated 
non-abelian free group. Let $F$ be a finitely
generated free group. Theorem \ref{t:bm} tells us that 
the isomorphism problem is unsolvable among the finitely presented
subgroups of $\G_0\times\G_0\times F$, which is a subgroup of the
right angled Artin group $A\times A\times F$.
\qed

\section{Embedding Virtual RAAGs in Mapping Class Groups}\label{s:raags}

A {\em right angled Artin group} (RAAG) is a group given by a presentation of the form
$$
A = \< v_1,\dots,v_n \mid [v_i,v_j] = 1 \,\forall (i,j)\in E\>.
$$
Thus $A$ is encoded by the finite graph with vertex set $\{1,\dots,n\}$ and edge set $E$. Such
groups arise naturally throughout mathematics whenever one has some automorphisms of an object,
some of which commute but which are otherwise unrelated. One such setting is that of surface automorphisms:
if two simple closed curves on a surface are disjoint, then the Dehn twists in those curves commute,
but if one has a set curves, no pair of which can be homotoped off each other, then high powers of
the twists in those curves freely generate a free group. 
(Significantly sharper results of 
this sort are proved in \cite{koberda} and \cite{CLM}.)
 It follows that any RAAG can be embedded in the
mapping class group of any surface $S$ of sufficiently high genus: it suffices that the dual of the
graph defining $A$ can be embedded in $S$. (This is explained by
Crisp and Wiest in \cite{CW}; I do not know if it appears in the literature earlier.)

The surface $S$ need not be closed. Indeed if one has a RAAG $A$
generated by powers of some Dehn twists, then
one can puncture the surface or delete discs in the complement of the curves supporting the
generators. The pattern of intersections of the curves on the modified surface $S'$ is unaltered,
so the commuting relations between the generating twists remain and the 
subgroup that they generate $A'<\mods'$ is isomorphic to $A$,
because the restriction of the natural homomorphism $\mods\to\mod$ provides 
an inverse to the homomorphism $A\to A'$ defined by sending the generating twists to themselves.

Thus one can embed an arbitrary RAAG in the mapping class group of a compact surface with boundary.
To embed virtual RAAGs we appeal to the following construction from \cite{mb:ddh}.

Recall that the wreath product 
$A\wr B$ of groups is the semidirect product $B\ltimes\oplus_{b\in B} A_b$, where the
$A_b$ are isomorphic copies of $A$ permuted by left translation. If $K<H$ is a normal subgroup
of finite index, then there is a standard embedding of $H$ into $K\wr (H/K)$, so if one has
a representation of $K$ into the mapping class group of a surface with boundary then one can obtain
an action of $K\wr (H/K)$ as follows. 

\begin{prop}\label{l:sg} If $S$ is a compact surface with non-empty boundary and 
$G$ is a finite group, then there is a closed surface $S_g$ and a monomorphism
${\rm{Mod}}(S)\wr G \to {\rm{Mod}}(S_{g})$.
\end{prop}
 
\begin{proof} Let
 $\overline S$ be the surface obtained from $S$ by attaching a 1-holed torus to
 all but one of the boundary components of $S$ (so $\overline S$ has just one boundary component). 
Every finite group has a faithful realisation 
as a group of symmetries of a closed hyperbolic surface.
We realise $G$ on a surface and equivariantly delete an open
disc about each point in a free orbit. We then glue a copy of $\overline S$
to each of the  resulting boundary components  and extend
the action of $G$ in the obvious manner. Let $S_g$ be the resulting surface.

Corresponding to each of the attached copies of $\overline S$ there is an
injective homomorphism ${\rm{Mod}}(S) \to {\rm{Mod}}(S_{g})$ 
obtained by extending
homeomorphisms to be the identity on the complement of the attached copy
of $S$. (The extension respects isotopy classes.) Thus we obtain
$|G|$ commuting copies $\{M_\g\mid \g\in G\}$
 of ${\rm{Mod}}(S)$ in ${\rm{Mod}}(S_{g})$. The
action of $G\subset {\rm{Mod}}(S_{g})$ by conjugation permutes the $M_g$
and the canonical map $G\ltimes \oplus_{\g\in G} M_\g\to {\rm{Mod}}(S_{g})$
 is injective.
\end{proof} 

\begin{corollary}\label{c:fi} If a group $\G$ has a subgroup of finite index that embeds in a right angled Artin group, then
$\G$ embeds in ${\rm{Mod}}(S_{g})$ for infinitely many closed surfaces $S_g$.
\end{corollary}

\begin{remark} There is considerable flexibility in the above construction, but
in general one cannot embed $\G$ in $\mod$ for {\em all}
surfaces of sufficiently high genus. Indeed there are constraints
even for finite groups \cite{ravi}.
\end{remark}

\section{Subgroups with infinitely many conjugacy classes of torsion elements}\label{s:fin}

Problem 3.10  on Farb's list \cite{farb} asks for a finitely presented subgroup of a mapping class group that has
finitely many conjugacy classes of finite subgroups. (Mapping class groups themselves have only finitely many conjugacy
classes of finite subgroups -- see the appendix to \cite{mb:sl}.) This can be settled by combining Proposition \ref{l:sg} with
 the constructions used in \cite{mb:sl} to provide the first examples of finitely presented subgroups of ${\rm{SL}}(n,\Z)$
 that have infinitely many conjugacy classes of finite subgroups, as we shall explain.
 
 \begin{prop}\label{p:surffin} There are infinitely many closed surfaces whose mapping class groups contain finitely presented subgroups with
 infinitely many conjugacy classes of finite cyclic subgroups. 
 \end{prop}

We give two explicit examples of subgroups as described in the preceding proposition.
Earlier examples are due to Brady, Clay and Dani \cite{BCD}.
 
 \begin{example} Let $T_n = \<a,b\mid [a,b]^n=1\>$, where $n\ge 2$. This is a cocompact lattice in  ${\rm{PSL}}(2,\R)$.
 The hyperbolic orbifold $\mathbb H^2/T_n$  is a torus with a  cone point of index $n$, corresponding to the fixed point of $[a,b]$.
  For each epimorphism $T_n\to\Z$,
 the kernel of the induced map from $D=T_n\times T_n\times T_n\times T_n$ onto $\Z$ is finitely presented and has infinitely many conjugacy
 classes of elements of order $n$ (see \cite{mb:sl} Proposition 4). $D$ has a subgroup of finite index that is a direct product of 
 four copies of the fundamental group of a closed hyperbolic surface. There are many ways to embed a
surface group in the mapping class group of a compact surface with boundary, so we can apply the
construction of Proposition \ref{l:sg}.
\end{example} 
 
 \begin{example} Consider the subgroup of finite index $\G_0(17)<{\rm{SL}}(2,\Z)$ consisting of matrices such that $17$ divides the
 entry in the bottom left corner. It is well known that there is a surjection $\phi:\G_0(17)\to\Z$. 
In $\G_0(17)$ one has the following
 element of order four:
 $$
\gamma=\begin{pmatrix}  
      4 & -1 \\
      17&-4\\
 \end{pmatrix}.
$$ 
Let $n\ge 3$, let $\Phi_n:\G_0(17)\times\dots\times\G_0(17)\to\Z$ be the surjection that restricts to $\phi$ on each factor,
and let $K_n$ be the kernel.   
It is proved on page 630 of \cite{mb:sl} that $K_n$  is finitely presented,
and among the cyclic subgroups $\<(x^{-1}\gamma x, \gamma,\dots,\gamma)\>$ with $x\in\G_0(17)$, there
are infinitely many conjugacy classes in $K_n$. It is easy to embed a direct product of $g$ non-abelian free groups in the mapping
class group of a (closed, punctured or bounded) surface of genus $g$, with each summand supported on a subsurface homeomorphic
to a torus with one boundary component. Thus, using Proposition \ref{l:sg}
 or otherwise, one can embed a direct product
of any number of copies of ${\rm{SL}}(2,\Z)$, which is virtually free,
 in the mapping class group of a closed surface.
\end{example}

\begin{remark}\label{r:LN}
Although the examples constructed above are finitely presented, they are not of type ${\rm{FP}}_\infty$. 
In \cite{LN} Leary and Nucinckis construct a group $\G$ and a subgroup
of finite index $H$ so that $H$ has a finite classifying space and is a subgroup of a RAAG, but $\G$ has infinitely 
many conjugacy classes of finite subgroups. Later results of Leary and Hsu \cite{ijl}, \cite{HL} establish the
corresponding result for finite cyclic subgroups. Thus, in the light of Corollary \ref{c:fi}, one can improve the condition ``finitely presented" in 
Proposition \ref{p:surffin}  to ``type VF".
\end{remark}

\section{Dehn functions}

As RAAGs and mapping class groups are automatic \cite{mosher}, one can solve the word problem in such a group in quadratic time \cite{Ep}.
This solution to the word problem can be applied to any finitely generated subgroup, and hence such subgroups also have a solvable
word problem. But this leaves open the question of how complicated the geometry of the word problem for finitely presented
subgroups can be (as measured by their Dehn functions, for example).

Recall that the {\em Dehn function} of a finitely presented group $\G=\<A\mid R\>$ measures the complexity of the word problem by counting the
number of times one has to apply the defining relations in order to prove that a word $w$ in the generators represents the identity in the group:
${\rm{Area}}(w)$ is defined to be the least integer $N$ for which there is an equality
$$
w = \prod_{i=1}^N \theta_i r_i^{\pm 1}\theta_i^{-1}
$$ in the free group $F(A)$, with $r_i\in R$, and the Dehn function of $\<A\mid R\>$ is 
$$
\delta(n) := \max \{{\rm{Area}}(w) \mid w=_\G 1,\,\ |w|\le n\},
$$
where $|w|$ denotes word-length. 

It is an open question as to whether all finitely presented subgroups of direct products of free groups have polynomial Dehn functions (cf.~\cite{dison}),
but as a further illustration of our main theme -- that subgroups of RAAGs are much more diverse -- we offer:

\begin{prop} There exist right-angled Artin groups $A$ and finitely presented subgroups $P<A$
such that $P$ has an exponential Dehn function.
\end{prop}

\begin{proof} Let $M$ be a finite-volume hyperbolic 3-manifold that fibres over the circle. Then
$\pi_1M = \Sigma\rtimes\Z$, and $\G=\pi_1M\times\pi_1M$ contains $P:= (\Sigma\times\Sigma)\rtimes\Z$,
the inverse image of the diagonal in $\G/(\Sigma\times\Sigma) = \Z\times\Z$. The Dehn function of $P$ is exponential; see
\cite{mb:haef} Theorem 2.5. The growth of a Dehn function is preserved on passage to subgroups of finite
index, and by \cite{agol} there is a subgroup of finite index in $\G$ that embeds 
in a right-angled Artin group.
\end{proof}

\subsection{Farb \cite{farb}, Question 3.9} This question
 asks if all finitely presented subgroups of $\mod$ are combable or automatic. The answer to this question is {\em{no}}.
Indeed one can already see this by observing that $\mod$ contains a direct product of $g$ non-abelian
free groups, and if $g\ge 3$ then such a direct product contains finitely presented groups whose
 third homology is not finitely generated
\cite{stall}, whereas the homology groups of combable and automatic groups are finitely generated in every dimension \cite{Ep}. A natural
relaxation of Farb's question would be to ask if all finitely presented subgroups of mapping class groups have quadratic, or at least polynomial,
Dehn functions. The above construction shows that the answer to this weaker question is still {\em{no}}.

\begin{corollary} 
The mapping class group of any surface of sufficiently high genus contains a 
finitely presented subgroup whose Dehn function is exponential.
\end{corollary}


\end{document}